\documentclass[reqno,a4paper,12pt]{amsart} 

\usepackage{amsmath,amscd,amsfonts,amssymb}
\usepackage{mathrsfs,dsfont}

\usepackage{tikz}
\usetikzlibrary{patterns}
\usepackage{float}
\usepackage{appendix}
\usepackage{caption}
\usepackage{subcaption}
\usepackage{enumerate}

\numberwithin{equation}{section}
\numberwithin{figure}{section}

\def\R{\mathbb{R}}

\def\Z{\mathbb{Z}}

\def\1{\mathds{1}}

\renewcommand\le{\leqslant}
\renewcommand\ge{\geqslant}
\renewcommand\leq{\leqslant}
\renewcommand\geq{\geqslant}

\renewcommand\hat{\widehat}
\renewcommand\Re{\operatorname{Re}}
\renewcommand\Im{\operatorname{Im}}

\newcommand{\supp}{\operatorname{supp}}

\newcommand{\dist}{\operatorname{dist}}

\newcommand{\Ai}{\operatorname{Ai}}

\theoremstyle{plain}
\newtheorem{thm}{Theorem}[section]
\newtheorem{lem}[thm]{Lemma}

\newtheorem{prop}[thm]{Proposition}

\newtheorem*{claim*}{Claim}
\newtheorem*{thm*}{Theorem}

\theoremstyle{definition}

\newtheorem*{definition*}{Definition}
\newtheorem*{remarks*}{Remarks}
\newtheorem*{remark*}{Remark}

\newenvironment{enumerate-math}
{\begin{enumerate}
\addtolength{\itemsep}{5pt}
}
{\end{enumerate}}

\newenvironment{enumerate-text}
{\begin{enumerate}
\addtolength{\itemsep}{5pt}
}
{\end{enumerate}}

\begin{document}

\title[oscillatory integrals and
Fourier frames]{Lower bounds for parameter-dependent oscillatory integrals and Fourier frames on hypersurfaces}

\author{Shilei Fan}
\address{School of Mathematics and Statistics, and Key Lab NAA–MOE, Central China Normal University, Wuhan 430079, China}
\email{slfan@ccnu.edu.cn}

\author{Bochen Liu}
\address{Department of Mathematics \& International Center for Mathematics, Southern University of Science and Technology, Shenzhen 518055, China}
\email{Bochen.Liu1989@gmail.com}

\thanks{Bochen Liu was partially supported by the National Key R\&D Program of China 2024YFA1015400. Shilei Fan was partially supported by the NSFC (grants No.12331004). The authors acknowledge the use of generative AI tools for suggesting the Airy function example in Section 5, constructing the example in Section 6, and checking the proof steps. All examples and proofs have been independently verified by the authors and presented in their own words. The authors bear full responsibility for all content of this work.}
\date{}


\begin{abstract}
We prove lower bounds for Fourier transforms of surface measures
$$|\widehat{\psi\,d\sigma}(\xi)|  \gtrsim_{\psi, K, \sigma, R_0} \sigma(B(p, |\xi|^{-1})), \text{ given }p\in K\Subset\{\psi>0\} ,\ \vec{n}_{p}\parallel \xi,\ |\xi|\ge R_0,$$
on convex hypersurfaces of finite type with at most one vanishing principal curvature
at each point. These bounds are uniform in the choice of $p\in K\Subset \{\psi>0\}$
and match the classical upper bound of Bruna-Nagel-Wainger.
The key new analytic ingredient is a uniform lower bound for parameter-dependent oscillatory integrals on $\mathbb{R}$ with convex phases. We also give convex examples of finite type with two vanishing principal
curvatures for which 
$$\inf_{p\in K}|\widehat{\psi\,d\sigma}(R \vec{n}_p)|\lesssim_N R^{-N},\ \text{ for arbitrary fixed } \{0\}\Subset K\Subset\{\psi>0\}.$$

Though the estimate above has its own interest, our original motivation is to study Fourier frames for surface measures. In particular, we show that, if $\Gamma$ is a closed convex smooth surface in $\R^d, d\geq 2$, of finite type and at most one principal curvature vanishes at each point, then the surface measure on $\Gamma$ does not admit any Fourier frame. This weakens the non-vanishing curvature condition in a previous result of Iosevich, Lai, the second author, and Wyman. On centrally symmetric surfaces of non-vanishing Gaussian curvature, we obtain more domains without Fourier frames. For example, we show that $\{x\in S^{d-1}:0<x_d<\delta\}, d\geq 3$, does not admit any Fourier frame for all $\delta>0$, which generalizes the result of Chen and the second author on the hemisphere.
\end{abstract}
\maketitle

\section{Introduction}
\subsection{Fourier transform of surface measures}
Suppose $\sigma$ is the surface measure on a smooth surface $\Gamma\subset\R^d$, and $\psi\in C_0^\infty(\Gamma)$ is a nonnegative smooth cutoff function such that no pair of points in $\supp\psi$ have parallel normals. If the direction of $\xi$ stays a fixed positive distance from all normal directions over $\supp \psi$, then the Fourier transform decays rapidly:
\begin{equation}
\label{Fourier-fast-decay}	
|\widehat{\psi\,d\sigma}(\xi)|\leq C_N|\xi|^{-N}.
\end{equation}
If this surface has non-vanishing Gaussian curvature everywhere, then the standard stationary phase method gives
\begin{equation}
\label{Fourier-upper-bound}	
|\widehat{\psi\,d\sigma}(\xi)|\lesssim  (1+|\xi|)^{-\frac{d-1}{2}}, \forall\, \xi\in\R^d.
\end{equation}
Both \eqref{Fourier-fast-decay} and \eqref{Fourier-upper-bound} are classical estimates in harmonic analysis, and have been used widely in different areas of mathematics. There are several different proofs. We refer to, for example, \cite{Ste93}\cite{Sogge17}\cite{Wol03} as references. In particular, the proof in \cite{Wol03} relies on the following asymptotic expansion.
\begin{prop}[Proposition 6.4 in \cite{Wol03}]\label{prop-stationary-phase}
Suppose $a\in C_{0}^{\infty}(\R^n)$ is supported in a sufficiently small neighborhood of the origin,  $h\in C^{\infty}(\R^n)$, and assume that $\nabla h(0) = 0$ and the Hessian matrix $H_h(0)$ is invertible. 
Let $s_0$ be the signature of $H_h(0)$, $\Delta := 2^{-n} |\det (H_h(0))|$, and
\[
I(\lambda) := \int e^{-\pi i \lambda h(x)} a(x) \, dx.
\]
Then, as $\lambda\rightarrow+\infty$, for any $N$,
\[
I(\lambda) = e^{-\pi i \lambda h(0)} e^{-\pi i \frac{s_0}{4}} \Delta^{-\frac{1}{2}} 
\lambda^{-\frac{n}{2}} \left( a(0) + \sum_{j=1}^{N} \lambda^{-j} \mathcal{D}_{j, h} a(0) 
+ \mathcal{O}_h(\lambda^{-(N+1)}) \right).
\]
Here $\mathcal{D}_{j,h}$ are certain differential operators of order $2j$, with coefficients depending on $h$, and the implicit constant depends on $h$ and bounds for finitely many derivatives of $a$.
\end{prop}

Proposition \ref{prop-stationary-phase} actually implies that, for every compact $ K\subset \{\psi>0\}$ there exists $R_0>0$ such that 
\begin{equation}\label{approx-equal}|\widehat{\psi\,d\sigma}(\xi)|\approx |\xi|^{-\frac{d-1}{2}}, \forall\,\xi/|\xi|\in G^{\pm}(K), |\xi|>R_0.\end{equation}
Here and throughout this paper, $G$ denotes the Gauss map. Also denote $G^\pm(K):=G(K)\cup(-G(K))$, so the orientation does not matter. If $\Gamma$ is convex, we assume $G(p)$ is the outer normal at $p\in \Gamma$.

To see why \eqref{approx-equal} holds, notice that for every $p\in \Gamma$, by the Morse lemma one can change variables such that $p\mapsto 0$ and a neighborhood of $p$ becomes the graph of
$$x_1^2\pm\cdots\pm x_n^2$$
near the origin. If the Gaussian curvature is not zero at $p$, this change of variables can be taken to be smooth in $p$. So the remainder in Proposition \ref{prop-stationary-phase} is uniform near $p$.

Unfortunately, this stability does not hold if the Gaussian curvature vanishes. For example, although a standard computation (see, e.g., Proposition 3,  Chapter VIII in \cite{Ste93}) shows that
\begin{equation}\label{expansion-x-3}\begin{aligned}\int e^{i\lambda x^3} a(x)dx = &a(0)\int_{-\infty}^{\infty}e^{i\lambda x^3}dx + a'(0)\int_{-\infty}^{\infty}x e^{i\lambda x^3}dx+\cdots\\=& a(0)\cdot \frac{\Gamma(1/3)}{\sqrt{3}}\,\lambda^{-1/3}+O(\lambda^{-2/3}),\quad \lambda\to+\infty, \end{aligned}\end{equation}
 the estimate \eqref{approx-equal} with $\lambda^{-1/3}$ in place of $\lambda^{-\frac{d-1}{2}}$ fails for the graph $(x, x^3)$ on arbitrarily small compact neighborhoods $K$ of the origin. See Section \ref{sec-eg-convex} for a detailed discussion. The idea behind this example is, as $x^3$ is not convex at the origin, the two stationary contributions can cancel for normals $e\in S^1\backslash\{(0, \pm 1)\}$ arbitrarily close to the north pole. So, to generalize \eqref{approx-equal} to more general surfaces, a convexity assumption is necessary. 

We say a smooth surface is of finite type if every tangent line at every point has finite order of contact with the surface. We refer to \cite{BNW88} for equivalent definitions.

On smooth convex surfaces $S$ of finite type, the celebrated work of Bruna-Nagel-Wainger \cite{BNW88} proves the optimal upper bound 
\begin{equation}\label{upper-bound-BNW88}|\widehat{\psi\,d\sigma}(\xi)|\lesssim \sigma(B(p_0, |\xi|^{-1})),\end{equation}
given the outer normal $\vec{n}_{p_0}\parallel \xi$, where $B(p_0,\delta)$ is the cap cut out by the tangent slab at $p_0$ of width $\delta$: 
\begin{equation}\label{def-B}B(p_0,\delta):=\left\{p\in S:\operatorname{dist}(p,T_{p_0}S)<\delta
\right\}=\{p\in S: |\langle p-p_0,\vec{n}_{p_0}\rangle|<\delta\}\end{equation}
contained in a neighborhood of $p_0$ in which no pairs of points share parallel normals.

A very important point on the upper Fourier bound \eqref{upper-bound-BNW88} is that the implicit constant is uniform in $p_0$. In this paper, we wonder if there exists a matching uniform lower bound. As a remark, if a normal direction is fixed, Schulz \cite{Schulz91} obtained an asymptotic expansion of $\widehat{\psi\,d\sigma}(\xi)$ for convex hypersurfaces of finite type, while, as pointed out in Remark 4 in \cite{Schulz91}, the constants and the threshold in the resulting lower bound
may depend on the direction, thus not uniform.

Now we state our results. We first show that, for convex curves of finite type in the plane, the lower Fourier bound matches the upper bound of Bruna-Nagel-Wainger. In fact, we prove a more general theorem on parameter-dependent oscillatory integrals.
\begin{thm}
    \label{thm-reverse}
    Let $\Theta$ be a set of parameters. For each $\theta\in\Theta$, let $a_\theta\in C_0^\infty((-1,1))$ be real-valued with $\inf_\theta a_\theta(0)>0$, and $q_\theta\in C^\infty([-1,1])$ be convex satisfying 
\begin{equation}\label{eq:q-theta}q_\theta(0)=q_\theta'(0)=0, q_\theta(x)\geq c_q |x|^M, \ M\geq 2,\, x\in[-1,1],\end{equation}
with $c_q>0$ and $M\geq 2$ independent of $\theta$. Suppose 
\begin{equation}\label{eq:uniform-in-theta}\max_{0\leq n\leq M+1}\|q_\theta^{(n)}\|_{L^\infty([-1,1])}, \max_{n=0,1}\|a_\theta^{(n)}\|_{L^\infty}\lesssim 1\end{equation}
uniformly over $\theta\in\Theta$. Then there exist $c, \lambda_0>0$, independent of $\theta$, such that
\begin{equation}\label{eq:oned-conclusion}
  \left|\int e^{-i\lambda q_\theta(x)}a_\theta(x) dx\right|
  \ge c|\{x: |q_\theta(x)|\leq \lambda^{-1}\}|, \quad \forall\,\theta\in\Theta,\ \lambda\ge\lambda_0.
\end{equation}
In particular,
\begin{equation}\label{lower-bound-needed}\left|\int e^{-i\lambda q_\theta(x)}a_\theta(x) dx\right|
  \gtrsim \lambda^{-\frac{1}{2}},\ \forall\lambda\geq \lambda_0,\end{equation}
where the implicit constant is independent of $\theta$.
\end{thm}

By combining Theorem \ref{thm-reverse} and the stationary phase method, we obtain the desired Fourier lower bound for surfaces on which at most one principal curvature vanishes at each point.
\begin{thm}
    \label{thm-lower-Fourier-bound}
    Suppose $\Gamma$ is a smooth convex surface embedded in $\R^d, d\geq 2$, of finite type, and at most one principal curvature vanishes at each point. Let $\sigma$ be its surface measure and $G^{\pm}$ be as in \eqref{def-G}. Then every
$p_0\in \Gamma$ has a relatively open neighborhood $U$ such that, for every nonnegative $\psi\in C_0^\infty(U)$ and every compact  $K\subset \{\psi>0\}$, there are constants $c, R_0>0$, depending only on the surface,
$K$, and $\psi$, such that
\begin{equation}
  |\widehat{\psi\,d\sigma}(\xi)|
  \ge c\cdot \sigma(B(p, |\xi|^{-1})),
  \quad \forall\,p\in K,\ \vec{n}_{p}\parallel \xi,\ |\xi|\ge R_0.
\end{equation}
In particular,
$$ |\widehat{\psi\,d\sigma}(\xi)|
  \gtrsim |\xi|^{-\frac{d-1}{2}},
  \quad \forall\,\xi/|\xi|\in G^{\pm}(K),\ |\xi|\ge R_0.$$
\end{thm}

We give examples in Section \ref{sec-eg-higher-dimension} with two vanishing principal curvatures for which the localized Fourier transform decays faster than every power along a sequence of normal frequencies. These examples also show that the corresponding
higher-dimensional analogue of Theorem \ref{thm-reverse} fails. Technically, what makes curves different from surfaces is: for nonzero convex polynomials $Q$ on $\R$ with $Q(0)=0$ and $\nabla Q(0)=0$,
$$\Im \int_{-\infty}^\infty e^{iQ(x)}dx>0$$
always holds (see \cite{BNW88}, Section 5, Remark 2), which is crucial to our proof of Theorem \ref{thm-reverse} (see \eqref{lim-positive} below). However, this may not hold with two variables (see \eqref{eq:local-model-zero} below).

\subsection{Fourier frames for surface measures}
Our motivation to study the lower bound of $|\widehat{\psi\,d\sigma}(\xi)|$ is on Fourier frames for surface measures. Suppose $\mu$ is a finite Borel measure on $\R^d$. A discrete set $\Lambda\subset\R^d$ is called a frame
spectrum of $\mu$ if there exist $0<A\leq B<\infty$ such that
$$A\|f\|_{L^2(\mu)}^2\leq \sum_{\lambda\in\Lambda}|\hat{f\,d\mu}(\lambda)|^2\leq B\|f\|_{L^2(\mu)}^2,\ \forall\,f\in L^2(\mu).$$
The notion of Fourier frame was introduced by Duffin and Schaeffer \cite{DS52} in 1952, and has become a powerful tool in applied harmonic analysis. It is shown by Nitzan-Olevskii-Ulanovskii \cite{NOU16} that every measurable subset of the Euclidean space of finite Lebesgue measure admits Fourier frames (with $d\mu=\chi_Edx$). Because of this, discussions on the non-existence of Fourier frames are mainly on singular measures. For related work, we refer to \cite{Lai11}\cite{HLL13}\cite{DL14}\cite{FL18}\cite{LL25+}\cite{KL25+}\cite{DLT26+}.

In dimensions two and higher, a natural class of singular measures is surface measures. In \cite{Lev18}, Lev pointed out that the surface measure on a small spherical cap admits Fourier frames, and then asked about the whole sphere. This question was answered negatively by Iosevich, Lai, Wyman and the second author in \cite{ILLW19}. More generally, they proved that the surface measure on closed smooth convex surfaces with nonvanishing curvature everywhere does not admit any Fourier frame. In contrast, they also proved that the surface measure on the boundary of every polytope admits Fourier frames.

Recently, the non-existence result on the whole sphere was generalized by Chen and the second author \cite{CL25} to every hemisphere. Although their result works on general centrally symmetric surfaces, non-vanishing Gaussian curvature everywhere is still required.

In this paper, a surface $\Gamma$ is a finite union of the images of smooth embeddings $\chi_i$ from open balls $ B(0,1)\subset\R^{d-1}$ to $\R^d$. We say $\Gamma$ is convex if each $\chi_i$ is locally convex. Let $\Gamma_i:=\chi_i(B(0,1))$, $G_i: \Gamma_i\mapsto S^{d-1}$ denote the Gauss map,
\begin{equation}\label{def-G}G_i^\pm(\cdot):=G_i(\cdot\cap\Gamma_i)\cup (-G_i(\cdot\cap\Gamma_i)) \text{ 
and }G^{\pm}(\Gamma):=\cup G_i^\pm(\Gamma_i).\end{equation}
Each $\Gamma_i$ admits its surface measure, whose sum, divided the multiplicity when overlapping, is the surface measure on $\Gamma$. We assume $\Gamma$ has finite surface measure. Then we say that $\Gamma$ admits Fourier frames if its surface measure does. 
Although some results below remain valid for self-intersecting surfaces, for convenience we assume $\Gamma$ is an embedded surface.

We improve previous results from two perspectives.

The first is to allow the Gaussian curvature to vanish.

\begin{thm}\label{thm-main-1}
Suppose $\Gamma$ is a smooth convex surface in $\R^d$ of finite type, at most one principal curvature vanishes at each point, and
\begin{itemize}
    \item $G^{\pm}(\Gamma)=S^1$, if $d=2$, or
    \item $G^{\pm}(\Gamma)$ contains a neighborhood of an equator in $S^{d-1}$, if $d\geq 3$.
\end{itemize}
Then the surface measure on $\Gamma$ does not admit any Fourier frame.
\end{thm}
As an application, if $\Gamma$ is a closed convex smooth surface in $\R^d$ of finite type and at most one principal curvature vanishes at each point of $\Gamma$, then the surface measure on $\Gamma$ does not admit any Fourier frame. This generalizes the result in \cite{ILLW19} to a more general class of surfaces on which Gaussian curvature may vanish. 

It is necessary to state the cases $d=2$ and $d\geq 3$ separately in Theorem \ref{thm-main-1}. For example, one can find a direction along which the projection of a small neighborhood of $\{(\pm 1, 0)\}$ in $S^1$ consists of two separated intervals, thus the arc-length measure on these two arcs admit Fourier frames by Lev's observation \cite{Lev18}.

We next consider smooth halves of centrally symmetric hypersurfaces. If a surface $\Gamma$ is centrally symmetric about the origin, then one can consider half-$\Gamma$s, denoted by $\Gamma/2$, that are smooth surfaces contained in $\Gamma$ whose surface measure satisfies
$$\sigma_{\Gamma/2}+\sigma_{-\Gamma/2}=\sigma_{\Gamma}.$$

\begin{thm}\label{thm-main-2}
	Suppose $\Gamma$ is a smooth centrally symmetric surface in $\R^d, d\geq 3$, of non-vanishing Gaussian curvature, centered at the origin, $0\notin\Gamma$, and $G^{\pm}(\Gamma)$ contains a neighborhood of an equator in $S^{d-1}$. Then the surface measure on every half-$\Gamma$ does not admit any Fourier frame.
\end{thm}
As an example, $\{x\in S^{d-1}:0<x_d<\delta\}, d\geq 3$, does not admit any Fourier frame for every $\delta>0$, which generalizes the result in \cite{CL25} on the hemisphere.

Theorems \ref{thm-main-1}, \ref{thm-main-2} are both corollaries of the following general theorem on Fourier-analytic properties of $\Gamma$.
\begin{thm}\label{thm-main-3}
	Suppose $\Gamma$ is a smooth surface in $\R^d$ whose Gauss map is an open mapping, and the following two conditions hold:
    \begin{enumerate}[(i)]
        \item \begin{itemize}
        	\item $G^{\pm}(\Gamma)=S^1$ when $d=2$, or 
        	\item $G^{\pm}(\Gamma)$ contains a neighborhood of an equator in $S^{d-1}$, when $d\geq 3$.
        \end{itemize}
        \item for every point $p\in\Gamma$, there exists a nonnegative $\psi\in C_0^\infty(\Gamma)$ with $\psi(p)=1$ such that, for every compact $ K\subset \{\psi>0\}$ there  exists $C, R_0>0$ such that
        \begin{equation}\label{goal-lower-bound}\int_{|\eta|\leq C}|\widehat{\psi\,d\sigma}(\xi+\eta)|^2d\eta \gtrsim_{\Gamma, \psi, C, R_0}  |\xi|^{-(d-1)}, \forall\, \xi/|\xi|\in G^{\pm}(K), |\xi|>R_0.\end{equation}
    \end{enumerate}
     Then the surface measure on $\Gamma$ does not admit any Fourier frame.
    
    Moreover, if $\Gamma$ is centrally symmetric about the origin, $0\notin\Gamma$, satisfying $(i)$, and (ii) holds with $\Re \widehat{\psi\,d\sigma}$ in place of $\widehat{\psi\,d\sigma}$, namely 
    \begin{equation}\label{lower-bound-Re}\int_{|\eta|\leq C}|\Re \widehat{\psi\,d\sigma}(\xi+\eta)|^2d\eta \gtrsim_{\Gamma, \psi, C, R_0}  |\xi|^{-(d-1)}, \forall\, \xi/|\xi|\in G^{\pm}(K), |\xi|>R_0,\end{equation}
    then every half-$\Gamma$ does not admit any Fourier frame.
\end{thm}

It is easy to conclude that the Gauss map is open when $\Gamma$ has non-vanishing Gaussian curvature or is convex of finite type. Then Theorem  \ref{thm-main-3} and Theorem \ref{thm-lower-Fourier-bound} together immediately imply Theorem \ref{thm-main-1}. 
For Theorem \ref{thm-main-2}, choose $\psi$ supported in a
sufficiently small chart and fix a compact
$K\subset\{\psi>0\}$. By Proposition \ref{prop-stationary-phase},
\[
\Re\widehat{\psi\,d\sigma}(R\theta+\eta)
=R^{-(d-1)/2}A(\theta)
 \cos\bigl(\beta(R,\theta)-2\pi p_\theta\cdot\eta\bigr)
+O(R^{-(d+1)/2}),
\]
uniformly for $\theta\in G^\pm(K)$ and $|\eta|\le1$.
Here $p_\theta\in K$ is the stationary point,
$\beta(R,\theta)\in[0,\pi]$ is real, and $\inf_\theta A(\theta)>0$.
Since $K$ is compact and $0\notin K$, compactness yields
\[
\inf_{\substack{p\in K\\ \beta\in\mathbb R}}
\int_{|\eta|\le1}|\cos(\beta-2\pi p\cdot\eta)|^2\,d\eta>0.
\]
Overall,
\[
\int_{|\eta|\le1}
|\Re\widehat{\psi\,d\sigma}(R\theta+\eta)|^2\,d\eta
\gtrsim R^{-(d-1)}
\]
for all sufficiently large $R$, uniformly in $\theta$.
Thus \eqref{lower-bound-Re} holds, and
Theorem \ref{thm-main-3} implies Theorem \ref{thm-main-2}.

\subsection*{Organization.} The paper is organized as follows. In Section \ref{sec-pf-thm-1.2}, we prove Theorem \ref{thm-reverse}. In Section \ref{sec-pf-lower-Fourier-bound}, we use Theorem \ref{thm-reverse} and the stationary phase method to prove Theorem \ref{thm-lower-Fourier-bound}. In Section \ref{sec-pf-thm-1.6}, we prove Theorem \ref{thm-main-3}. In Section \ref{sec-eg-convex}, we discuss the Fourier transform of the arc length measure on $(x,x^3)$ near $x=0$. In Section \ref{sec-eg-higher-dimension}, we construct and discuss an example with two vanishing principal curvatures.


\subsection*{Notation.}
Throughout this article, we write $A\lesssim B$ if there exists an absolute constant $C>0$ that only depends on the ambient surface and given cutoff functions such that $A\leq CB$.

\section{Proof of Theorem \ref{thm-reverse}}\label{sec-pf-thm-1.2}
Before we start our proof, we would like to remind the readers that,  in \cite{BNW88}, Section 5, Remark 2, it is proved that: if $P(t)=\sum_{j=2}^m a_j t^j$ is a convex polynomial on $[0,\infty)$ which is not identically zero, then
\begin{equation}\label{lower-bound-BNW88}
\left|\int_0^\infty e^{i\lambda P(t)}\,dt\right|\gtrsim
\mu(\lambda,a_2,\ldots,a_m),
\end{equation}
 where $\mu\in(0,\infty)$ is the unique solution to
$$\sum_{j=2}^m |\lambda a_j|\mu^j=1.$$
Although it looks far from Theorem \ref{thm-reverse}, the argument behind it inspires several key steps below.

Now we start our proof.

We first collect some properties of $q_\theta$ that will be used later: by \eqref{eq:q-theta}\eqref{eq:uniform-in-theta} and the convexity, it follows that
\begin{equation}\label{eq:two-sided-height}
c_0|x|^M\le q_\theta(x)\le C_0|x|^2, \ \forall\,x\in(-1,1),
\end{equation}
and, with $q_\theta^+(x):=q_\theta(x), q_\theta^-(x):=q_\theta(-x)$,
\begin{equation}\label{lower-bound-q'}
(q^\pm_\theta)'(x)\geq (q^\pm_\theta)(x)/x \geq c_0 x^{M-1},\ \forall\,x\in(0,1).
\end{equation}

Denote
\[
  I_\theta(\lambda):=\int e^{-i\lambda q_\theta(x)}a_\theta(x)\,dx.
\]
Inspired by Remark after Theorem 5.1, 5.2 in \cite{BNW88}, we consider the imaginary part of $I_\theta(\lambda)$, namely
\[
  -\int a_\theta(x) \sin (\lambda q_\theta(x))\,dx.
\]

Let
\begin{equation}\label{def-delta}\delta_{\theta,\lambda}:=\max\{\delta^+_{\theta,\lambda}, \delta^-_{\theta,\lambda}\}\in(0,\infty), \text{ where }q^\pm_\theta(\delta^\pm_{\theta,\lambda})=\lambda^{-1}.\end{equation}
Here $\delta_{\theta,\lambda}$ is well defined for every $\theta$ when $\lambda>c_q^{-1}$ (recall \eqref{eq:q-theta}). And it follows from \eqref{eq:two-sided-height} that
\begin{equation}\label{eq:delta-bounds}
  C_0^{-1/2}\lambda^{-1/2}
  \le\delta_{\theta,\lambda}
  \le c_0^{-1/M}\lambda^{-1/M}.
\end{equation}
Also
$$\delta_{\theta,\lambda} \le |\{q_\theta\le\lambda^{-1}\}| =\delta_{\theta,\lambda}^++\delta_{\theta,\lambda}^- \le2\delta_{\theta,\lambda}.$$

Write
$$\int a_\theta(x) \sin (\lambda q_\theta(x))\,dx=\delta_{\theta,\lambda}\int a_\theta(\delta_{\theta,\lambda} x) \sin (\lambda q_\theta(\delta_{\theta,\lambda} x))\,dx.$$

Then, to prove this theorem, it suffices to show that there exists $\lambda_0>0$ such that
$$ \inf_{\theta\in\Theta, \lambda\geq \lambda_0}\left|\int a_\theta(\delta_{\theta,\lambda} x) \sin (\lambda q_\theta(\delta_{\theta,\lambda} x))\,dx\right|>0.$$
Suppose otherwise. Then there exist a sequence $\lambda_j\rightarrow\infty$ and pairs $(\lambda_j, \theta_j), j=1,2,\dots$, such that, with $\delta_j:=\delta_{\theta_j,\lambda_j}$,
\begin{equation}\label{tend-to-0}\int a_{\theta_j}(\delta_j x) \sin (\lambda_j q_{\theta_j}(\delta_j x))\,dx\rightarrow 0.\end{equation}
The main idea in the argument below is to find a convergent subsequence $\lambda_{j_k}q_{\theta_{j_k}}(\delta_{j_k} x)\rightarrow Q(x)$, where $Q$ is a convex polynomial of degree $\geq 2$ and $Q(0)=Q'(0)=0$. Then, intuitively,
$$a_{\theta_j}(0)^{-1}\int a_{\theta_j}(\delta_j x) \sin (\lambda_j q_{\theta_j}(\delta_j x))\,dx\rightarrow \int \sin Q(x)\,dx,$$
which is positive by the argument in Section 5, Remark 2, a contradiction.

Now we give a rigorous proof to contradict \eqref{tend-to-0}.

We first show that there exists a subsequence $j_k\rightarrow\infty$ such that 
$$\lambda_{j_k} q_{\theta_{j_k}}(\delta_{j_k} x)\rightarrow Q(x), \ \forall x\in\R,$$
where $Q$ is a polynomial of degree $\leq M$. 

Consider the Taylor expansion of $q_\theta$:
$$q_\theta(x)=\sum_{n=0}^M\frac{q_\theta^{(n)}(0)}{n!}x^n + R_\theta (x):=Q_\theta (x)+R_\theta (x).$$
For fixed $x\in\R$ and large $j$ such that $\delta_j |x|\leq 1$, the remainder of $\partial_x^n(\lambda_j q_{\theta_j}(\delta_jx))$ satisfies
\begin{equation}\label{polynomial-tail}|\partial_x^n(\lambda_j R_{\theta_j}(\delta_jx))|\leq \frac{\|q_\theta^{(M+1)}\|_{L^\infty}}{(M+1-n)!}\lambda_j\delta_j^{M+1}|x|^{M+1-n}\lesssim \lambda_j\delta_j^{M+1}|x|^{M+1-n},\ n=0,1,\dots, M,\end{equation}
uniformly in $\theta$ by \eqref{eq:uniform-in-theta}, thus $\rightarrow 0$ uniformly in $\theta$ by \eqref{eq:delta-bounds}. In particular, $|\partial_x^n(\lambda_j R_{\theta_j}(\delta_jx))|$ are uniformly bounded in $j$ and $x\in[-1,1]$.

Therefore, to find a convergent subsequence of $\lambda_j q_{\theta_j}(\delta_jx)$, it suffices to find a convergent subsequence of Taylor polynomials 
$$Q_j:=\lambda_j Q_{\theta_j}(\delta_jx)=\sum_{n=0}^M a_{j, n}x^n.$$
By the monotonicity of $q^\pm_\theta$ on $(0,1)$ and our choice of $\delta_{\theta,\lambda}$ in \eqref{def-delta}, we have
$$0\leq \min_\pm\lambda_j q^\pm_{\theta_j} (\delta_jx)\leq \min_\pm\lambda_j q^\pm_{\theta_j} (\delta_j)= 1, \ \forall\,x\in[0,1].$$
As the remainder $|\lambda_j R_{\theta_j}(\delta_jx)|$ is uniformly bounded in $j$ and $x\in[-1,1]$, it follows that 
$$\min\{\|Q_j\|_{L^\infty([-1,0])}, \|Q_j\|_{L^\infty([0,1])}\}\lesssim 1$$
 uniform in $j$. As $Q_j$s are polynomials of degree $\leq M$ and all norms on finite dimensional Banach spaces are equivalent, it follows that $$\|Q_j\|_{L^\infty([-1,0])}\approx \|Q_j\|_{L^\infty([0,1])}$$
 uniformly in $j$, and therefore $\|Q_j\|_{L^\infty([-1,1])}\lesssim 1$ uniformly in $j$. Then, since every linear operator between finite dimensional Banach spaces must be bounded, the linear operator
 $$(\sum_{n=0}^M a_n x^n, \|\cdot\|_{L^\infty([-1,1])})\rightarrow (a_0,\dots,a_M)_{l^\infty}$$ 
gives that
$$\max_{0\leq n\leq M}|a_{j,n}|\lesssim \|Q_j\|_{L^\infty([-1,1])}\lesssim 1$$
uniformly in $j$. So there exists a subsequence $j_k$ such that $\lim_{j_k\rightarrow\infty}a_{j_k,n}$ exists for all $n=0,\dots, M$. Hence, $ Q_{j_k}(x)$, thus $\lambda_{j_k} q_{\theta_{j_k}}(\delta_{j_k}x)$, is convergent to a polynomial $Q$ of degree $\leq M$ pointwise. In fact \eqref{polynomial-tail} implies that 
$$\partial^n(\lambda_{j_k} q_{\theta_{j_k}}(\delta_{j_k}x))\rightarrow \partial^n Q, n=0,\dots, M.$$

Notice that pointwise convergence preserves non-negativity, monotonicity, and convexity. Since $q_\theta(0)=q_\theta'(0)=0$ for all $\theta$, it follows that $Q(0)=Q'(0)=0$. Also, since $\min_\pm\{\lambda_j q_{\theta_j}(\pm\delta_j)\}= 1$ for all $j$, the polynomial $Q$ is not identically zero, thus strictly decreasing in $(-\infty, 0]$ and strictly increasing in $[0,\infty)$. Consequently, $\deg Q\geq 2$ and $|Q'(x)| \rightarrow \infty$ as $|x|\rightarrow \infty$.

Let $\{x_n\}_{n\in\mathbb{Z}}$ be an increasing sequence such that $Q(x_{\pm n})=2n\pi$, then $x_0=0$, $\lim_{n\rightarrow\infty} x_n=\infty$, $\lim_{n\rightarrow-\infty} x_n=-\infty$, and
$$\begin{aligned}\lim_{n\rightarrow\infty}\int_{x_{-n}}^{x_n}\sin Q(x)\,dx=&\lim_{n\rightarrow\infty }\left(\sum_{l=0}^{n-1} \int_{x_l}^{x_{l+1}} \sin Q(x)\,dx+\sum_{l=0}^{n-1} \int_{-x_{-l}}^{-x_{-(l+1)}} \sin Q(-x)\,dx\right)\\=&\lim_{n\rightarrow\infty}\left(\sum_{l=0}^{n-1} \int_{2l\pi}^{2(l+1)\pi}\sin (u) J_+(u)\,du+\sum_{l=0}^{n-1} \int_{2l\pi}^{2(l+1)\pi}\sin (v) J_-(v)\,dv\right),\end{aligned}$$
where $u=Q(x), v=Q(-x), x>0$ and $J_\pm$ are the Jacobians $$J_+(u)=\frac{\partial x}{\partial u}=\frac{1}{Q'(x)},\ J_-(v)=\frac{\partial x}{\partial v}=\frac{1}{-Q'(-x)}.$$
This limit can be further written as
$$\lim_{n\rightarrow\infty}\left(\sum_{l=0}^{n-1} \int_{2l\pi}^{(2l+1)\pi}\sin (u) (J_+(u)-J_+(u+\pi))\,du+\sum_{l=0}^{n-1} \int_{2l\pi}^{(2l+1)\pi}\sin (v) (J_-(v)-J_-(v+\pi))\,dv\right).$$
As $Q$ is strictly decreasing in $(-\infty, 0)$ and strictly increasing in $(0,\infty)$, so is $Q^{-1}$. Also $Q'$ is increasing over $\R$ due to the convexity of $Q$, and non-constant because $Q$ is a nonzero polynomial with $Q(0)=Q'(0)=0$. Overall, it follows that $J_\pm(u)$ are positive decreasing functions and each integral in the sum above is positive. Hence the monotone convergence theorem guarantees that
\begin{equation}\label{lim-positive}\lim_{n\rightarrow\infty}\int_{x_{-n}}^{x_n}\sin Q(x)\,dx:=\epsilon_0\in(0, \infty].\end{equation}
We shall show that this positivity contradicts \eqref{tend-to-0}. For the integral in \eqref{tend-to-0}, its tail near $\infty$ equals
$$\begin{aligned}
&\int_{x_n}^\infty a_{\theta_j}(\delta_j x) \sin (\lambda_j q_{\theta_j}(\delta_j x))\,dx\\=&\frac{a_{\theta_j}(\delta_j x_n) \cos (\lambda_j q_{\theta_j}(\delta_j x_n))}{\lambda_j\delta_j q'_{\theta_j}(\delta_j x_n)}+\int_{x_n}^\infty \left(\frac{\delta_j a'_{\theta_j}(\delta_j x)}{\lambda_j\delta_j q'_{\theta_j}(\delta_j x)}-\frac{a_{\theta_j}(\delta_j x)(\delta_j q'_{\theta_j}(\delta_j x))'}{\lambda_j(\delta_j q'_{\theta_j}(\delta_j x))^2}\right)\cos (\lambda_j q_{\theta_j}(\delta_j x))dx.	
\end{aligned}
$$
As $q'_\theta>0$ and $q''_\theta\geq 0$ in $(0, 1)$, its absolute value is
$$\leq \frac{1}{\lambda_j\delta_j q'_{\theta_j}(\delta_j x_n)}+\int_{x_n}^{1/\delta_j} \frac{\delta_j |a'_{\theta_j}(\delta_j x)|}{\lambda_j\delta_j q'_{\theta_j}(\delta_j x)}dx+\int_{x_n}^{1/\delta_j}\frac{(\delta_j q'_{\theta_j}(\delta_j x))'}{\lambda_j(\delta_j q'_{\theta_j}(\delta_j x))^2}dx\lesssim \frac{1}{\lambda_j\delta_j q'_{\theta_j}(\delta_j x_n)},$$
uniform in $\theta_j$. The same estimate holds on the tail of \eqref{tend-to-0} near $-\infty$ because $(q^-_\theta)'>0$ and $(q^-_\theta)''\geq 0$ hold in $(0, \infty)$. Therefore, as $\lambda_j\delta_j q'_{\theta_j}(\delta_j R)\rightarrow Q'(R)$ and $\deg Q\geq 2$, we have
\begin{equation}\label{tail-j}\lim_{n\rightarrow\infty}\limsup_{j_k\rightarrow\infty }\left|\int_{\R\backslash[x_{-n}, x_n]} a_{\theta_{j_k}}(\delta_{j_k} x) \sin (\lambda_{j_k} q_{\theta_{j_k}}(\delta_{j_k} x))\,dx\right|\lesssim \lim_{n\rightarrow\infty}\frac{1}{|Q'(x_{\pm n})|}=0.\end{equation}

Combining \eqref{lim-positive} and \eqref{tail-j}, we can choose $x_{n_0}$ independent of $j$ such that
\begin{equation}\label{>-epsilon/2}\int_{x_{-n_0}}^{x_{n_0}}\sin Q(x)\,dx>\min\{1, \epsilon_0\}/2\end{equation}
while
\begin{equation}\label{<-epsilon/4}\limsup_{j_k\rightarrow\infty }|\int_{\R\backslash[x_{-{n_0}}, x_{n_0}]} a_{\theta_j}(\delta_j x) \sin (\lambda_j q_{\theta_j}(\delta_j x))\,dx|<\inf_\theta a_\theta(0)\cdot \min\{1, \epsilon_0\}/4.\end{equation}
This is impossible because
$$\begin{aligned}\left|\int_{x_{-n_0}}^{x_{n_0}} \sin Q(x)\,dx\right|=&\lim_{j_k\rightarrow \infty} \left|\int_{x_{-n_0}}^{x_{n_0}} \frac{a_{\theta_{j_k}}(\delta_{j_k} x)}{a_{\theta_{j_k}}(0)} \sin (\lambda_{j_k} q_{\theta_j}(\delta_{j_k} x))\,dx\right|\\\leq & \frac{1}{\inf_\theta a_{\theta}(0)}\cdot\limsup_{j_k\rightarrow \infty }\left|\int_{x_{-n_0}}^{x_{n_0}} a_{\theta_{j_k}}(\delta_{j_k} x) \sin (\lambda_{j_k} q_{\theta_{j_k}}(\delta_{j_k} x))\,dx\right|\\\leq &\min\{1, \epsilon_0\}/4,\end{aligned}$$
a contradiction to \eqref{>-epsilon/2}, where the last line follows from \eqref{tend-to-0} and \eqref{<-epsilon/4}.

\section{Proof of Theorem \ref{thm-lower-Fourier-bound}}\label{sec-pf-lower-Fourier-bound}
We may assume $d\geq 3$, as the case $d=2$ follows directly from Theorem \ref{thm-reverse}.

For each $p_0\in\Gamma$, translate $p_0$ to the origin and rotate coordinates so
that a neighborhood of $p_0$ is parametrized by
\[
 \Phi(x,t)=(x,t,h(x,t)),
 \qquad x\in U\subset (-10,10)^{d-2},\quad t\in (-10,10),
\]
where $h$ is convex, $h(0)=0$, $\nabla h(0)=0$. As the second fundamental form is given by the matrix $\frac{H h}{\sqrt{1+|\nabla h|^2}}$, we can write
$$Hh=\begin{pmatrix}
    D^2_xh(x,t) & \nabla_x h_{t}\\ (\nabla_x h_{t})^T & h_{tt}
\end{pmatrix}$$
and assume that $D^2_xh\approx I_{d-2}$.

By a limiting argument, the contact order in a neighborhood of the origin is $\leq M$. Then, by compactness, we have $$\inf_{(x,t)\in [-1, 1]^{d-1}, \theta\in S^{d-2}} \sum_{n=2}^M |D_\theta^n h(x,t)|>0,$$
which implies (see Section 4 in \cite{BNW88})
\begin{equation}\label{contact-order-M}|h(x,t)-h(x_0, t_0)
   -\nabla h(x_0,t_0)\cdot(x-x_0, t-t_0)|\geq c_0|(x-x_0, t-t_0)|^M\end{equation}
for all $(x_0,t_0)\in (-1, 1)^{d-1}$ and $(x,t)\in (-1, 1)^{d-1}$.

We only consider $\xi=(\xi_1,\dots,\xi_d)=|\xi|\theta$ with $\theta\in S^{d-1}$ near the north/south pole, so $|\xi|\approx |\xi_d|$. Denote $\lambda:=\xi_d$ and
$$\xi =\lambda (v, 1)=\lambda (v', v_{d-1}, 1).$$
As $\Gamma$ is convex of finite type, we may assume that for each $\xi=|\xi_d|(v,1)$ with $\xi/|\xi|\in G^\pm (K)$ near the north pole there exists a unique $(x_\theta,t_\theta)$ at which the normal 
\begin{equation}\label{normal-v}(\nabla h(x,t), -1)=(-v,-1).\end{equation} 
Then
$$\widehat{\psi\,d\sigma}(\xi)=\iint 
  e^{-2\pi i\lambda (v\cdot (x,t)+h(x,t))}b(x,t)\,d x\, d t,$$
  where $b(x,t)\in C_0^\infty$ is nonnegative and $\inf_{\theta\in G(K)}b(x_\theta, t_\theta)\gtrsim 1$.

We shall fix $t$ and apply the stationary phase method to the integral over $x$. So we need to find the stationary point for the phase function $v'\cdot x+h(x,t)$. In other words, we need to solve for $v'+\nabla_x h=0$. Due to the uniform positive definiteness of $D^2_x h$, we may assume this stationary point is unique.

Because $D^2_xh\approx I_{d-2}$ and $\nabla_x h(0,0)=0$, one can solve for \begin{equation}\label{implicit-function-thm}\nabla_x h(x(v', t),t)+v'=0\end{equation}
near the origin by the implicit function theorem. We may assume $x(v',t)$ is smooth in $(v',t)$ on $(-1,1)^{d-1}$. In particular, $x(v', t_\theta)=x_\theta$.

Let $\chi\in C_0^\infty((-1,1))^{d-1}$ be nonnegative with $\chi=1$ near the origin. We may assume $\chi=1$ on a neighborhood of all stationary points for $G^\pm(K)$. Denote $\tilde{b}:=b\cdot \chi $. Because of the fast decay of 
$$\iint 
  e^{-2\pi i\lambda (v\cdot (x,t)+h(x,t))}b(1 -\chi)(x,t)\,d x\, d t$$
by integration by parts, one can reduce $\widehat{\psi\,d\sigma}(\xi)$ to
\begin{equation}\label{Fourier-transform}I(\lambda, \theta)=\iint 
  e^{-2\pi i\lambda (v\cdot (x,t)+h(x,t))}\tilde{b}(x,t)\,d x\, d t.\end{equation}
The fixed compact support of the amplitude, the
uniform bounds on the required phase and amplitude derivatives, and
the uniform positive definiteness of $D_x^2h$ ensure that the
stationary phase remainder below is uniform in $\theta$.

Then, as for each $t\in(-1,1)$, the unique stationary point in $(-1,1)^{d-2}$ is $x(v',t)$, it follows that
\begin{equation}\label{apply-stationary-phase}\begin{aligned}|I(\lambda, \theta)|=&C_d\lambda^{-\frac{d-2}{2}}
   \left|\int e^{-2\pi i\lambda g_\theta(t)}a_\theta(t)\,d t\right|
   +O(\lambda^{-\frac{d-2}{2}-1})\\=&C_d\lambda^{-\frac{d-2}{2}}
   \left|\int e^{-2\pi i\lambda (g_\theta(t)-g_\theta(t_\theta))}a_\theta(t)\,d t\right|
   +O(\lambda^{-\frac{d}{2}}),\end{aligned}\end{equation}
where
\begin{equation} g_\theta(t)
 =v'\cdot x(v',t)+v_{d-1}t+h(x(v',t),t),
 \end{equation}
 and
\begin{equation}
 a_\theta(t)
 =\frac{\tilde{b}(x(v', t),t)}{\sqrt{\det D_x^2h(x(v',t),t)}}.
\end{equation}
Here the error term $O(\lambda^{-\frac{d}{2}})$ is uniform in $\theta$ because derivatives of $g_\theta, a_\theta$ are uniformly bounded in $\theta$.

Now we check properties of the phase function
$$q_\theta(t):=g_\theta(t)-g_\theta(t_\theta).$$
Recall $x(v', t_\theta)=x_\theta$ and $\nabla h(x_\theta, t_\theta)=-v$ from \eqref{normal-v}. Therefore
\begin{equation}\label{expression-q-theta}q_\theta(t)=h(x(v',t),t)-h(x_\theta,t_\theta)-\nabla h(x_\theta, t_\theta)((x(v',t),t)-(x_\theta,t_\theta))\end{equation}
that satisfies $q_\theta(t_\theta)=0$ and 
$$q'_\theta(t)=\nabla_xh\cdot x_t +h_t-\nabla h(x_\theta, t_\theta)\cdot (x_t, 1)=h_t(x(v',t),t)-h_t(x_\theta, t_\theta)$$
due to \eqref{normal-v}. Hence $q'_\theta(t_\theta)=0$ and 
$$q''_\theta(t)=\nabla h_t\cdot (x_t,1).$$
To compute $x_t$, we differentiate both sides of \eqref{implicit-function-thm} to obtain $x_t=-(D^2_x h)^{-1} \cdot \nabla_x h_t$. Therefore 
$$q''_\theta(t)=h_{tt}- (\nabla_x h_t)^T\cdot (D^2_x h)^{-1} \cdot \nabla_x h_t$$
which is the Schur complement of $D_x^2 h$ in $Hh$, thus $\geq 0$.

Finally, by \eqref{contact-order-M} and the expression of $q_\theta$ from \eqref{expression-q-theta}, we have
$$|q_\theta(t)|\geq c_0 |(x(v',t),t)-(x_\theta, t_\theta)|^M\geq c_0|t-t_\theta|^M.$$

Then we change variables 
\[
  \bar q_\theta(u)=q_\theta(t_\theta+u),\qquad
  \bar a_\theta(u)= a_\theta(t_\theta+u)
\]
and apply Theorem~\ref{thm-reverse}. Notice the uniformity condition \eqref{eq:uniform-in-theta} follows because derivatives of $g_\theta, a_\theta$ are uniformly bounded in $\theta$. Hence there exists $\lambda_0>0$ independent of $\theta$ such that, for all $\lambda>\lambda_0$,
$$\begin{aligned}|I(\lambda, \theta)|\gtrsim &\lambda^{-\frac{d-2}{2}}|\{ t: q_\theta(t) \leq \lambda^{-1}\}|+O(\lambda^{-\frac{d}{2}})\\
\gtrsim &\lambda^{-\frac{d-2}{2}}|\{t: q_\theta(t)\leq \lambda^{-1}\}|,\end{aligned}$$
where the implicit constant is independent of $\theta$. Here $O(\lambda^{-d/2})$ is an error term because $|\{t: q_\theta(t)\leq \lambda^{-1}\}|\gtrsim \lambda^{-1/2}$.

It remains to show
\[
\lambda^{-\frac{d-2}{2}}
|\{t:q_\theta(t)\leq \lambda^{-1}\}|
\gtrsim
\sigma\bigl(B(p_\theta,\lambda^{-1})\bigr).
\]

Recall that, for each fixed $t$,
\[
D_x^2 h \approx I_{d-2}
\]
uniformly in $\theta$. Hence, by the Morse lemma, after a
smooth change of variables in the $x$-variables, with Jacobian uniformly
bounded above and below, the height above the tangent plane at
$p_\theta$ is locally of the form
\[
q_\theta(t)+|y|^2.
\]
Therefore, by Fubini's theorem,
$$\sigma(B(p_\theta,|\xi|^{-1}))
\approx
\int \max\{0, \lambda^{-1}-q_\theta(t)\}^{\frac{d-2}{2}}\,dt\lesssim \lambda^{-\frac{d-2}{2}}|\{t:q_\theta(t)\leq \lambda^{-1}\}|.$$

\section{Proof of Theorem \ref{thm-main-3}}\label{sec-pf-thm-1.6}

Since $G^\pm(\Gamma)$ has nonempty interior, by Sard's thoerem there must exist some $\Gamma_i\subset\Gamma$ and  some $p_0\in G_i^{-1}(D)\subset\Gamma_i$ on which the Gaussian curvature is nonzero. We may assume that $G_i(p_0)=(1,0,\dots,0)$. Therefore, by \eqref{Fourier-upper-bound}, there exists $\phi\in C_0^\infty(\Gamma)$ with 
$$G^\pm(\supp\phi)\subset \{|x|=1: |(x_2,\dots,x_{d})|\leq \delta/2\}\subset D, \ 0<\delta<1/100,$$
such that $$|\widehat{\phi\,d\sigma}(\xi)|\lesssim  (1+|\xi|)^{-\frac{d-1}{2}}, \forall\, \xi\in\R^d.$$

For $d=2$, the conclusion on $\Gamma$ follows directly from Theorem 1.3, 1.4 in \cite{ILLW19}, and the conclusion on half-$\Gamma$ follows from the same argument in \cite{CL25}. We omit details.

From now $d\geq 3$.

Assume the frame spectrum $\Lambda$ exists for the surface measure $\sigma_\Gamma$ on $\Gamma$. A well known property of $\Lambda$ is
\begin{equation}\label{convergence-lambda}\sum_{\lambda\in\Lambda\backslash\{0\}} |\lambda|^{-\alpha}<\infty,\ \forall\alpha>d.\end{equation}
To see this, notice that there exists $\epsilon>0$ such that $|\hat{\sigma_\Gamma}|\geq C>0$ on $B(0, \epsilon)$, then for all $x\in\R^d$,
\begin{equation}\label{separation}B\sigma(\Gamma)\geq\sum_{\lambda\in\Lambda}|\hat{\sigma_\Gamma}(\lambda-x)|^2\geq C^2\#\{\Lambda\cap B(x,\epsilon)\},\end{equation}
which implies that
$$\sum_{\lambda\in\Lambda\backslash\{0\}} |\lambda|^{-\alpha}\lesssim 1+\sum_{\vec{n}\in\Z^d\backslash\{0\}} |\vec{n}|^{-\alpha}<\infty,\ \forall\alpha>d.$$

As $G^{\pm}(\Gamma)$ contains a neighborhood of an equator, we may assume 
$$D:=\{|x|=1: |x_d|< 10\delta\}\subset\{|x|=1: |x_d|\leq 100\delta\}\subset G^{\pm}(\Gamma)$$
with $0<\delta<1/100$. Let $\mathcal{C}_D$ denote the cone generated by $D$, namely
$$\mathcal{C}_D:=\{tx:x\in D, t\in\R\}.$$

\subsection{Convergence of $\sum|\lambda|^{-d+1}$ in $\mathcal{C}_D$}\label{subsec-con}

 As $\overline{D}$ is a compact subset of $S^{d-1}$, one can find a finite family $\{\psi_j\}$ from $(ii)$ in Theorem \ref{thm-main-3}, such that, for each $j$, there exists an open set $D_j\subset S^{d-1}$ with $\overline{D_j}\Subset G^\pm(\{\psi_j>0\})$ satisfying
 $$\overline{D}\subset\bigcup_j D_j\cup (-D_j).$$ 
Then, by \eqref{goal-lower-bound} and the upper frame bound, we have 
$$\sum_{\lambda \in \Lambda\cap \mathcal{C}_D}(1+|\lambda|)^{-(d-1)}\leq \sum_j\sum_{\lambda \in \Lambda\cap \mathcal{C}_{D_j}}(1+|\lambda|)^{-(d-1)}\lesssim \sum_j\sum_{\lambda\in\Lambda\cap \mathcal{C}_{D_j}}\int_{|\eta|\leq C}|\hat{\psi_j\,d\sigma}(\lambda+\eta)|^2d\eta\lesssim 1.$$

\subsection{Divergence of $\sum|\lambda|^{-d+1}$ in $\mathcal{C}_D$}\label{subsec-div}
Recall $p_0$ and $\phi$ from the beginning of this section, and denote 
$$\mathcal{C}_\phi:=\{\xi\in\R^d:|(\xi_2,\dots,\xi_d)|\leq \delta|\xi|\}.$$


We shall need the following geometric lemma.
\begin{lem}\label{geo-lem}
 For all $R\geq 1$, there exists $\xi^R$ in the plane $x_d=0$ such that
$$\mathcal{N}_{2\delta R}(\mathcal{C}_\phi-\xi^R)\subset \mathcal{C}_D.$$	
\end{lem}
\begin{proof}
	Take 
	$$\xi^R:=(0, 2R,0,\dots,0).$$
	For each $\xi\in \mathcal{N}_{2\delta R}(\mathcal{C}_\phi-\xi^R)$, there exists $\xi^0\in \mathcal{C}_\phi$ such that $|\xi-(\xi^0-\xi^R)|\leq 2\delta R$.

If $|\xi^0|\leq R$, then 
$$|\xi|\geq |\xi^R|- |\xi^0|-|\xi-(\xi^0-\xi^R)|\geq (1-2\delta)R.$$ 
As $\xi^0\in\mathcal{C}_\phi$, we have $|\xi^0_d|\leq \delta |\xi^0|\leq \delta R$. Therefore, due to $\xi^R_d=0$,
$$|\xi_d|\leq |\xi_d^0|+2\delta R\leq 3\delta R\leq \frac{3\delta}{1-2\delta}|\xi|\leq 10\delta |\xi|.$$

If $|\xi^0|\geq R$, then, due to $\xi^R_1=0$,
$$|\xi|\geq |\xi_1|\geq  |\xi^0_1|-2\delta R\geq (1-\delta)|\xi^0|-2\delta R\geq (1-3\delta)|\xi^0|,$$
and therefore
$$|\xi_d|\leq 3\delta |\xi^0|\leq \frac{3\delta}{1-3\delta}|\xi|\leq 10\delta |\xi|.$$
\end{proof}

Now we are ready to show the divergence of $\sum |\lambda|^{-d+1}$ for the contradiction. By the lower frame bound,
\begin{equation}\label{lower-frame-bound}1\lesssim R^{-d}\int_{B_{\delta R}}\sum_{\lambda \in \Lambda}|\widehat{\phi\,d\sigma_\Gamma}(\lambda+\xi^R+\xi)|^2d\xi.\end{equation}
Also notice that, by \eqref{Fourier-upper-bound}\eqref{Fourier-fast-decay}, we have
\begin{equation}\label{Fourier-decay-phi}|\widehat{\phi\,d\sigma_\Gamma}(\xi)|\lesssim  (1+|\xi|)^{-\frac{d-1}{2}}, \forall\, \xi\in\R^d,\end{equation}
and
\begin{equation}\label{fast-decay-phi}|\widehat{\phi\,d\sigma_\Gamma}(\xi)|\leq C_N|\xi|^{-N}, \forall \xi\notin\mathcal{C}_\phi.\end{equation}

Let $c_\delta>0$ be a small number that will be determined later. 

For each $\lambda \notin \mathcal{N}_{2\delta R}(\mathcal{C}_\phi-\xi^R)$ and $|\xi|\leq \delta R$, we have $\dist(\lambda+\xi^R+\xi, \mathcal{C}_\phi)\geq \delta R$. In particular, $\lambda+\xi^R+\xi\notin \mathcal{C}_\phi$ and $|\lambda+\xi^R+\xi|\geq \delta R$. Therefore, by \eqref{fast-decay-phi} and \eqref{separation},
$$R^{-d}\int_{B_{\delta R}}\sum_{\lambda \in \Lambda\backslash \mathcal{N}_{2\delta R}(\mathcal{C}_\phi-\xi^R)}|\widehat{\phi\,d\sigma_\Gamma}(\lambda+\xi^R+\xi)|^2d\xi\lesssim \sum_{|\vec{n}|\geq \delta R}|\vec{n}|^{-10d},$$
which is $\leq c_\delta$ when $R$ is large, say $R>R_0$.

Now we invoke the geometric lemma above. By Lemma \ref{geo-lem}, after removing those $\lambda$ in $(\mathcal{N}_{2\delta R}(\mathcal{C}_\phi-\xi^R))^c$, the remaining frequencies lies in
$$\mathcal{N}_{2\delta R}(\mathcal{C}_\phi-\xi^R)\subset\mathcal{C}_D.$$

Then, as we have shown the convergence of $\sum_{\lambda \in \Lambda\cap \mathcal{C}_D}(1+|\lambda|)^{-(d-1)}$ from Section \ref{subsec-con}, together with \eqref{Fourier-decay-phi}, one can take $R_0$ large such that for all $R>R_0$, 
$$R^{-d}\int_{B_{\delta R}}\sum_{\lambda \in \Lambda\cap \mathcal{C}_D\cap B_R^c}|\widehat{\phi\,d\sigma_\Gamma}(\lambda+\xi^R+\xi)|^2d\xi\lesssim \sum_{\lambda \in \Lambda\cap \mathcal{C}_D\cap B_R^c}(1+|\lambda|)^{-(d-1)}\leq c_\delta.$$

Combining these estimates, if we choose $c_\delta$ much smaller than the left hand side of \eqref{lower-frame-bound}, we end up with
$$1\lesssim R^{-d}\int_{B_{\delta R}}\sum_{\lambda \in \Lambda\cap \mathcal{C}_D\cap B_R}|\widehat{\phi\,d\sigma_\Gamma}(\lambda+\xi^R+\xi)|^2d\xi,$$
which implies, as $\lambda+\xi^R+\xi\in B_{4R}$ always holds,
$$1\lesssim R^{-d}\int_{B_{4R}}|\widehat{\phi\,d\sigma_\Gamma}|^2\cdot \#\{\Lambda\cap \mathcal{C}_D\cap B_{4R}\}.$$
By \eqref{Fourier-decay-phi}, we conclude that

$$\#(\Lambda\cap \mathcal{C}_D\cap B_R)\gtrsim R^{d-1},$$
which implies
$$\sum_{\lambda \in \Lambda\cap \mathcal{C}_D}(1+|\lambda|)^{-(d-1)}= (d-1)\int_0^\infty (1+r)^{-d}\cdot \#(\Lambda\cap \mathcal{C}_D\cap B_r) \,dr=\infty,$$
a contradiction to the convergence from Section \ref{subsec-con}.

\subsection{Centrally symmetric $\Gamma$}
Let $\sigma=\sigma_\Gamma$ denote the surface measure on $\Gamma$, and $\sigma_{\Gamma/2}:=\sigma|_{\Gamma/2}$. Assume that $\Lambda$ is a frame spectrum for $\sigma_{\Gamma/2}$.

Let $D, D_j, \psi_j, p_0$ be as in the previous section. We may further assume $p_0\in\Gamma/2$. If not, one can perturb $p_0$ to some $p_0'\in\Gamma/2$ and slightly rotate $\Gamma$ such that $D\subset G^\pm(\Gamma)$ still holds and $G_i(p_0')=(1,0,\dots,0)$. Because of this, if one can show the convergence of $\sum|\lambda|^{-d+1}$ in $\mathcal{C}_D$, then the divergence, thus the contradiction, follows in the same way as Section \ref{subsec-div}.

It remains to show 
$$\sum_{\lambda \in \Lambda\cap \mathcal{C}_D}(1+|\lambda|)^{-d+1}<\infty.$$

Thanks to the central symmetry, if we take
$$\tilde{\psi_j}(x):=\psi_j(x)+\psi_j(-x),$$
then
$$\begin{aligned}2\Re \hat{\tilde{\psi_j}\,d\sigma_{\Gamma/2}}(\xi)= & \int_{\Gamma/2} \tilde{\psi_j}(x) e^{-2\pi i x\cdot \xi}\,d\sigma+ \int_{\Gamma/2} \tilde{\psi_j}(x) e^{2\pi i x\cdot \xi}\,d\sigma\\=&\int_{\Gamma/2} \tilde{\psi_j}(x) e^{-2\pi i x\cdot \xi}\,d\sigma+ \int_{-\Gamma/2} \tilde{\psi_j}(x) e^{-2\pi i x\cdot \xi}\,d\sigma\\=&\int_{\Gamma} (\psi_j(x)+\psi_j(-x)) e^{-2\pi i x\cdot \xi}\,d\sigma\\=&\int_{\Gamma} \psi_j(x) e^{-2\pi i x\cdot \xi}\,d\sigma+\int_{\Gamma} \psi_j(x) e^{2\pi i x\cdot \xi}\,d\sigma\\=&2\Re \hat{\psi_j\,d\sigma_\Gamma}(\xi).\end{aligned}$$
In particular,
\begin{equation}\label{compare-real-part}|\hat{\tilde{\psi_j}\,d\sigma_{\Gamma/2}}|\geq |\Re \hat{\tilde{\psi_j}\,d\sigma_{\Gamma/2}}|=|\Re \hat{\psi_j\,d\sigma_\Gamma}|,\end{equation}
so \eqref{lower-bound-Re} can help. Let $C, R_0$ be constants from there. 

A geometric observation is that, there exist $R'_0>R_0$, $D_j'\subset D_j$, such that $D\subset \cup D_j'$ and $\lambda+\xi\in \mathcal{C}_{D_j}$ for all $\lambda\in \mathcal{C}_{D_j'}\cap B_{R'_0}^c$ and $|\xi|\leq C$.

Finally, by \eqref{compare-real-part},\eqref{lower-bound-Re}, and the upper frame bound with $f=e^{-2\pi i x\cdot \xi}\tilde{\psi_j}(x)$, it follows that
$$1\gtrsim \sum_{\lambda\in \Lambda \cap \mathcal{C}_{D_j'}\cap B_{R'_0}^c } \int_{|\xi|\leq C}|\hat{\tilde{\psi_j}\,d\sigma_{\Gamma/2}}(\lambda+\xi)|^2\,d\xi\gtrsim \sum_{\lambda\in \Lambda \cap \mathcal{C}_{D_j'}\cap B_{R'_0}^c} |\lambda|^{-d+1}.$$

Summing over $j$, it follows that
$$\sum_{\lambda \in \Lambda\cap \mathcal{C}_D}(1+|\lambda|)^{-d+1}<\infty,$$
as desired.

\section{A non-convex example}\label{sec-eg-convex}
In this section, we show that \eqref{goal-lower-bound} fails on the curve $(x, x^3)$ near the origin.

Suppose $a\in C_0^\infty(\R)$ and $a=1$ near the origin. For $\lambda_1=\alpha \lambda_2^{1/3}$, where $\alpha>0$ is a constant that will be clarified later, consider
\begin{equation}\label{Airy}\begin{aligned}F(\lambda_1, \lambda_2):=&\int e^{i(-\lambda_1t+\lambda_2 t^3/3)} a(t)dt\\=& \lambda_2^{-1/3}\int e^{i (-\lambda_1\lambda_2^{-1/3}t+t^3/3)} a(\lambda_2^{-1/3}t)dt.\end{aligned}\end{equation}
Lemma 5.1 in \cite{BGGIST07} shows that an expansion like \eqref{expansion-x-3} still holds if $\alpha>0$ is small enough. We shall show that there exists a constant $\alpha_0>0$ such that the expected lower bounds in \eqref{expansion-x-3},\eqref{goal-lower-bound} fail with $\lambda_1=\alpha_0\lambda_2^{1/3}$.

Integration by parts shows that
$$\lambda_2^{-1/3}\int e^{i (-\lambda_1\lambda_2^{-1/3}t+t^3/3)} (1-a(\lambda_2^{-1/3}t))dt=O(\lambda_2^{-1}).$$
So $F(\lambda_1, \lambda_2)$ is reduced to
$$\lambda_2^{-1/3}\int e^{i (-\lambda_1\lambda_2^{-1/3}t+t^3/3)}\,dt=2\pi\lambda_2^{-1/3}\Ai(-\lambda_1\lambda_2^{-1/3}),$$
  where $\Ai(x)$ is the Airy function $$\text{Ai}(x):=\frac{1}{\pi}\int_0^\infty\cos(\frac{t^3}{3}+xt)\,dt,$$ that is known to be smooth and have negative real zeros. 
  
Let $-\alpha_0<0$ be a real zero of $\Ai$.  Then the above shows that
$$|F(\alpha_0\lambda_2^{1/3}, \lambda_2)|\lesssim \lambda_2^{-1}.$$
Therefore, for each $\lambda_1$ in a $\rho$-neighborhood of $\alpha_0\lambda_2^{1/3}$, the ratio $\lambda_1\lambda_2^{-1/3}$ lies in the $\rho\lambda_2^{-1/3}$ neighborhood of $\alpha_0$, which implies
$$|F(\lambda_1, \lambda_2)|=|F(\alpha_0\lambda_2^{1/3}, \lambda_2)|+O( \lambda_2^{-1/3}\cdot \rho\lambda_2^{-1/3})+O(\lambda_2^{-1})=O(\rho\lambda_2^{-2/3}).$$

Hence, for any $\epsilon>0$ and any $\rho>0$, there exists a sequence of $\rho$-balls $B(\vec{\lambda}_i, \rho)$ lying in the cone generated by the $\epsilon$-neighborhood of the north pole, such that
$$\lim_{i\rightarrow \infty}|\vec{\lambda}_i|^{4/3}\int_{B(\vec{\lambda_i}, \rho)}|F(\xi)|^2\,d\xi<\infty.$$
In particular, \eqref{goal-lower-bound} fails.

\section{An example with two vanishing principal curvatures}\label{sec-eg-higher-dimension}

In this section, we shall construct a smooth convex surface 
 $$\Gamma:=\{(x,y,P(x,y)):(x,y)\in\R^2\}$$
of finite type, $P(0)=0$, $\nabla P(0)=0$, such that every relatively open neighborhood $U$ of
the origin in $\Gamma$ admits a nonnegative $\psi\in C_0^\infty(U)$, a compact
neighborhood $K\subset\{\psi>0\}$ of the origin with $\psi>0$ on $K$, and
frequencies $\xi_\lambda$ with $\lambda\to+\infty$, satisfying
\begin{equation}\label{eq:local-conclusion}
 \frac{\xi_\lambda}{|\xi_\lambda|}\in G^{\pm}(K),\qquad
 |\xi_\lambda|\approx\lambda,\qquad
 |\widehat{\psi\,d\sigma}(\xi_\lambda)|\le C_N\lambda^{-N}.
\end{equation}
In dimensions $4$ and higher, one can just take the surface
$$\{(x_1,\dots,x_{d-1}, P(x_1, x_2)+x_3^2+\cdots +x_{d-1}^2)\}.$$ 

Now we start our construction. For each even integer $2m\geq 6$, let
\begin{equation}\label{eq:local-polynomial}
 P_m(x,y):=\left(x+\frac{y^2}{2m}\right)^{2m}
             +\frac{1}{2m}(x-y^2)^{2m}.
\end{equation}
Observe that $P(0)=0$, $\nabla P(0)=0$, and
\begin{equation}
	\label{lower-bound-P-m}
	P_m(x,y)\geq \frac{1}{2m}\left(|x|+\frac{y^2}{2m}\right)^{2m}\gtrsim_m |x|^{2m}+|y|^{4m}.
\end{equation}

We claim that $P_m$ is strictly convex, of
finite type, and for all sufficiently large $m$ there exist $(u_{m},v_{m})\in\R^2$ such that
\begin{equation}\label{eq:local-model-zero}
 F_{m}(u_m,v_m):=\lim_{\epsilon\downarrow0}
 \int_{\R^2}e^{-(\epsilon-i)P_m(x,y)-i(u_mx+v_my)}\,dx\,dy=0.
\end{equation}
We leave the proof to the end of this section. 

Fix such an $m=m_0$ and write $P:=P_{m_0}$, $(u_*,v_*):=(u_{m_0},v_{m_0})$, and take
\[
 \Gamma:=\{(x,y,P(x,y)):(x,y)\in\R^2\}.
\]
Under local coordinates near the origin, we can choose $\psi\in C_0^\infty(\Gamma)$ such that
\begin{equation}\label{from-psi-to-a}\psi\,d\sigma= a(x,y)\,dx\,dy,\end{equation}
where $a\in C_0^\infty$ and $a=1$ near the origin. We may assume $K$ is contained in the interior of $\{a=1\}$.

For $\lambda\ge1$, let
\begin{equation}\label{eq:local-frequencies}
 \xi_\lambda:=\frac1{2\pi}
 \bigl(u_*\lambda^{1/(2m_0)},v_*\lambda^{1/(4m_0)},-\lambda\bigr).
\end{equation}
Since $P=P_{m_0}$ is strictly convex, $P(0)=0$, $\nabla P(0)=0$, and the direction of $\xi_\lambda$ tends to the south pole as $\lambda\rightarrow +\infty$, when $\lambda$ is large enough there is a unique $(x_*,y_*)$
such that $\nabla P(x_*,y_*)\parallel (u_*,v_*)$. By rescaling
\eqref{eq:local-polynomial}, the point
\[
 p_\lambda:=\bigl(\lambda^{-1/(2m_0)}x_*,
     \lambda^{-1/(4m_0)}y_*,\lambda^{-1}P(x_*,y_*)\bigr)
\]
belongs to $\Gamma$ and tends to the origin. Therefore $
 G(p_\lambda)=\xi_\lambda/|\xi_\lambda|\in G^{\pm}(K)$ for all sufficiently large $\lambda$.

Now the change of variables
$(x,y)\mapsto(\lambda^{-1/(2m_0)}x,\lambda^{-1/(4m_0)}y)$ and
\eqref{eq:local-model-zero}\eqref{from-psi-to-a} give
\begin{equation}\label{eq:local-scaled-integral}
\begin{aligned}
 \widehat{\psi\,d\sigma}(\xi_\lambda)
 &=C\lambda^{-3/(4m_0)}
   \int e^{i(P(x,y)-u_*x-v_*y)}
   a(\lambda^{-1/(2m_0)}x,\lambda^{-1/(4m_0)}y)\,dx\,dy\\
 &=C\lambda^{-3/(4m_0)}
   \int e^{i(P(x,y)-u_*x-v_*y)}
   \bigl(a(\lambda^{-1/(2m_0)}x,\lambda^{-1/(4m_0)}y)-1\bigr)\,dx\,dy.
\end{aligned}
\end{equation}
Then we use the coordinates
\[
 (x,y)=(r\omega_1,r^{1/2}\omega_2),\qquad
 \omega\in\Sigma:=\{P=1\},\qquad
 dx\,dy=r^{1/2}\,dr\,d\nu(\omega),
\]
where $\Sigma$ is smooth and compact and $\nu$ is a smooth positive measure. Now the phase becomes
\[
 \phi(r,\omega)=r^{2m_0}-u_*r\omega_1-v_*r^{1/2}\omega_2,
\]
and integration by parts therefore gives the fast decay
\[
 |\widehat{\psi\,d\sigma}(\xi_\lambda)|
 \lesssim\lambda^{-3/(4m_0)}
 \int_{c\lambda^{1/(2m_0)}}^\infty r^{1/2-2m_0L}\,dr
 \lesssim_N\lambda^{-N}.
\]

\begin{proof}[The existence of $P_m$ and $(u_m,v_m)$]
The proof relies on the implicit function theorem on Banach spaces.

We first show the limit in \eqref{eq:local-model-zero} exists. For this purpose, define the entire function (recall \eqref{lower-bound-P-m})
\[
 \mathcal L_m(z_1,z_2)
 :=\int_{\R^2}e^{-P_m(x,y)-i(z_1x+z_2y)}\,dx\,dy.
\]
For positive real $s$, weighted dilation gives
\[
 \int e^{-sP_m(x,y)-i(ux+vy)}\,dx\,dy
 =s^{-3/(4m)}\mathcal L_m(s^{-1/(2m)}u,s^{-1/(4m)}v).
\]
Both sides are holomorphic in $\Re s>0$, with principal branches on the
right. Hence the identity holds throughout this half-plane. Letting
$s=\epsilon-i\to-i$ proves the existence of the limit in
\eqref{eq:local-model-zero} and gives
\begin{equation}\label{eq:local-complex-scaling}
 F_m(u,v)=e^{3\pi i/(8m)}
 \int e^{-P_m(x,y)}
 e^{-i(e^{i\pi/(4m)}ux+e^{i\pi/(8m)}vy)}\,dx\,dy.
\end{equation}

Then we show that
\begin{equation}\label{eq:local-C1-limit}
 F_m(u,v)\longrightarrow F_\infty(u,v)
 :=\int_{\{(x,y):y^2-1\le x\le1\}} e^{-i(ux+vy)}\,dx\,dy
\end{equation}
in $C^1$ on every compact rectangle of real parameters $(u,v)$.

Notice that $e^{-P_m}\rightarrow 1_{\{(x,y):y^2-1\le x\le1\}}$ almost everywhere. To justify convergence under the integral, notice that (consider $x\leq 0$, $0\leq 2x\leq y^2$, $2x\geq y^2$ separately)
$$|x|+y^2\leq 3\max\{|x+y^2/(2m)|, |x-y^2|\}.$$
So, when $|x|+y^2\geq 6$, the definition of $P_m$ in \eqref{eq:local-polynomial} implies that
$$P_m(x,y)\geq \frac{1}{2m}\left(\frac{|x|+y^2}{3}\right)^{2m}\gtrsim \left(\frac{|x|+y^2}{3}\right)^4,$$
where the implicit constant is independent of $m$. Therefore the $C^1$ convergence in \eqref{eq:local-C1-limit} follows from the dominated convergence theorem.

We next find a simple zero of $F_\infty$ in \eqref{eq:local-C1-limit}. Let
\begin{equation}\label{eq:local-limit-slice}
 A(v):=F_\infty(0,v)
 =\int_{-\sqrt2}^{\sqrt2}(2-y^2)e^{-ivy}\,dy
 =\frac{4\bigl(\sin(\sqrt2v)-\sqrt2v\cos(\sqrt2v)\bigr)}{v^3}.
\end{equation}
The choice $v_0=\rho/\sqrt2$, $\tan\rho=\rho$,
$\rho\in(\pi,3\pi/2)$, implies $A(v_0)=0$ and $A'(v_0)\ne0$.
Moreover,
\[
 \partial_uF_\infty(0,v)
 =-\frac i2\int_{-\sqrt2}^{\sqrt2}y^2(2-y^2)e^{-ivy}\,dy
 =\frac i2A''(v),
\]
and differentiating \eqref{eq:local-limit-slice} at $v_0$ gives
$A''(v_0)=-4A'(v_0)/v_0$. Thus
\begin{equation}\label{eq:local-simple-zero}
 \partial_uF_\infty(0,v_0)=-\frac{2i}{v_0}A'(v_0),\qquad
 \partial_vF_\infty(0,v_0)=A'(v_0).
\end{equation}

Now let $V$ be a bounded open rectangle containing $(0,v_0)$, let
$\mathcal K=\overline V$, and define
\[
 \Phi:C^1(\mathcal K;\R^2)\times V\longrightarrow\R^2,\qquad
 \Phi(h,u,v)=h(u,v)+\bigl(\Re F_\infty(u,v),\Im F_\infty(u,v)\bigr).
\]
This map is $C^1$: its derivative in $h$ is the evaluation map
$k\mapsto k(u,v)$, and
\[
 D_{u,v}\Phi
 =Dh(u,v)+
 \begin{pmatrix}
  \Re\partial_uF_\infty&\Re\partial_vF_\infty\\
  \Im\partial_uF_\infty&\Im\partial_vF_\infty
 \end{pmatrix}
\]
is continuous in $(h,u,v)$. By \eqref{eq:local-simple-zero},
$\Phi(0,0,v_0)=0$ and
\[
 \det D_{u,v}\Phi(0,0,v_0)
 =\det \begin{pmatrix}
  0&A'(v_0)\\
  -2A'(v_0)/v_0&0
 \end{pmatrix}=\frac{2}{v_0}A'(v_0)^2>0.
\]
The implicit function theorem on Banach spaces therefore gives a
neighborhood $\mathcal V$ of $0$ in $C^1(\mathcal K;\R^2)$ and a $C^1$ map
$(u(h),v(h)):\mathcal V\to V$ such that
\[
 (u(0),v(0))=(0,v_0) \text{ and }
 \Phi(h,u(h),v(h))=0,\ \forall\,h\in\mathcal V.
\]
By \eqref{eq:local-C1-limit}, we can take
\[
 h_m:=\bigl(\Re(F_m-F_\infty),\Im(F_m-F_\infty)\bigr)
 \longrightarrow0\quad\text{in }C^1(\mathcal K;\R^2).
\]
Hence, for all sufficiently large $m$, we may take
$(u_m,v_m):=(u(h_m),v(h_m))\to(0,v_0)$, and
\[
 0=\Phi(h_m,u_m,v_m)
   =\bigl(\Re F_m(u_m,v_m),\Im F_m(u_m,v_m)\bigr),
\]
which proves \eqref{eq:local-model-zero}.

It remains to verify convexity. Write $P_m(x,y)=H_m(x,y^2)$, where
\[
 H_m(x,t):=(x+\frac{t}{2m})^{2m}+\frac{1}{2m}(x-t)^{2m}.
\]
The function $H_m$ is convex because it is a sum of even powers of linear forms. Also
\[
 \partial_tH_m(x,t)
 =(x+\frac{t}{2m})^{2m-1}-(x-t)^{2m-1}\ge0\qquad(t\ge0).
\]
Therefore $P_m$ is convex. In fact, $\partial_tH_m(x,t)>0$ for $t>0$ and
$\partial_x^2H_m(x,t)>0$ unless $(x,t)=(0,0)$. The chain rule then shows
that $\nabla^2P_m$ is positive definite when $y\ne0$. On $y=0$,
$P_m(x,0)=(1+1/(2m))x^{2m}$ is strictly convex. Thus $P_m$ is strictly convex on $\R^2$. It is of finite type near the origin due to strict convexity and \eqref{lower-bound-P-m}.

\end{proof}

\bibliographystyle{plain}
\bibliography{mybibtex.bib}

\end{document}